\documentclass[11pt,reqno]{amsart}
\usepackage{fullpage}
\usepackage[mathscr]{eucal}
\usepackage{mathrsfs}
\usepackage{amsfonts}
\usepackage{amsmath}
\usepackage{color}
\usepackage[noautoscale,enableskew]{youngtab}
\usepackage{amsthm}
\usepackage{amssymb}
\usepackage{latexsym}
\usepackage[all]{xy}

\usepackage[mathscr]{eucal}
\usepackage{palatino}
\usepackage{fullpage}
\usepackage{verbatim}
\usepackage{xr-hyper}
\usepackage{hyperref}

\usepackage[dvipsnames]{xcolor}

\theoremstyle{plain}
\newtheorem{thm}[equation]{Theorem}
\newtheorem{cor}[equation]{Corollary}
\newtheorem{prop}[equation]{Proposition}
\newtheorem{lem}[equation]{Lemma}

\theoremstyle{definition}
\newtheorem{defn}[equation]{Definition}

\newtheorem{conj}[equation]{Conjecture}

\newtheorem{notation}[equation]{Notation}
\newtheorem{question}[equation]{Question}
\newtheorem{example}[equation]{Example}
\newtheorem{rem}[equation]{Remark}

\numberwithin{equation}{section}
\DeclareMathOperator{\Lie}{\mathrm{Lie}}
\DeclareMathOperator{\Tr}{\mathrm{Tr}}

\DeclareMathOperator{\Hilb}{\mathrm{Hilb}}

\newcommand{\beq}{\begin{equation}\label}
\newcommand{\eeq}{\end{equation}}

\DeclareMathOperator{\Spec}{\mathrm{Spec}}

\newcommand{\iso}{{\;\stackrel{_\sim}{\to}\;}}

\newcommand{\slt}{{\mathfrak{sl}}_2}
\newcommand{\sltt}{(\slt,T)}

\renewcommand{\o}{\otimes }

\renewcommand{\t}{{\mathfrak t}}

\newcommand{\la}{\lambda}

\newcommand{\h}{{{\mathfrak h}}}

\newcommand{\mc}{\mathcal}

\newcommand{\C}{\mathbb{C}}

\newcommand{\Z}{{\mathbb Z}}
\newcommand{\sset}{\subset}
\newcommand{\sminus}{\smallsetminus}
\newcommand{\into}{\,\hookrightarrow\,}

\newcommand{\mto}{\mapsto}

\DeclareMathOperator{\Irr}{\mathrm{Irr}}

\newcommand{\ds}{{\dots}}

\newcommand{\s}{\mathfrak{S}}

\newcommand{\Cs}{\C^\times}

\newcommand{\mf}{\mathfrak}

\newcommand{\Stab}{\mathrm{Stab}}

\newcommand{\X}{\mathfrak{X}}
\newcommand{\nc}{\newcommand}

\nc{\BM}{\mathbf{M}} 
\nc{\BG}{\mathbf{G}}
\nc{\BQ}{\mathbf{Q}}
\nc{\BL}{\mathbf{L}}
\nc{\BT}{\mathbf{T}}
\nc{\BH}{\mathbf{H}}

\nc{\LL}{{L'}}
\nc{\LSL}{\mathrm{L}}
\nc{\MSL}{\mathrm{M}}
\nc{\MM}{{M'}}
\nc{\tSL}{\mf{t}}
\nc{\TSL}{\mathrm{T}}

\nc{\To}{{T}^{\circ}}

\nc{\tcL}{\t_{\LGL}} 
\nc{\tsL}{\t^{\symL}} 
\nc{\tsLo}{(\t^{\symL})^{\circ}}
\nc{\tLo}{\t^{\circ}} 
\nc{\Xo}{\X^{\circ}}
\nc{\Lo}{L^{\circ}}
\nc{\Mo}{M^{\circ}}
\nc{\ZMo}{Z(M)^{\circ}}
\nc{\paX}{\partial \X}

\newcommand{\ms}{\mathscr}

\nc{\twist}{\mathrm{twist}}
\nc{\Span}{\mathrm{Span}}
\nc{\Ps}{\mathbb{P}}

\renewcommand{\H}{\mathsf{H}}

\nc{\Mod}{\mathrm{Mod} \,}
\nc{\ras}{\Lambda}
\nc{\Vect}{\mathrm{Vect}}

\newcommand{\sgn}{\mathrm{sgn}}

\newcommand{\vi}{${\sf {(i)}}\;$}
\newcommand{\vii}{${\sf {(ii)}}\;$}

\newcommand{\cX}{\mathsf{X}}

\renewcommand{\H}{\mathsf{H}}
\nc{\by}{\mathbf{y}}
\nc{\bx}{\mathbf{x}} 
\nc{\bz}{\mathbf{z}} 
\nc{\Y}{\mathcal{Y}}
\nc{\bY}{\overline{\Y}}

\nc{\IC}{\mathrm{IC}}

\nc{\Thetasph}{\Theta^{\mathrm{spher}}}

\nc{\Orb}{\mathbb{O}}

\nc{\bs}{\mathbf{s}}
\nc{\bp}{\mathbf{p}}
\nc{\bq}{\mathbf{q}}
\nc{\bc}{\mathbf{c}}
\nc{\br}{\mathbf{r}}
\nc{\Sing}{\mathrm{Sing}}

\nc{\cyc}{\mathrm{cyc}}
\nc{\ol}{\overline}
\nc{\abT}{\mathbb{T}}
\nc{\bw}{\mathbf{w}}
\nc{\mm}{\ms{M}}
\nc{\Cas}{\Pi}
\nc{\Jac}{\mathbb{J}}
\nc{\SSL}{{\operatorname{SL}}}
\nc{\gap}{{}_{}}
\nc{\stair}{\mathbf{m}}
\nc{\minusone}{-1}
\nc{\cY}{\mathsf{Y}}

\begin{document}

\title{{\textbf{$SL_2$-action on Hilbert schemes and Calogero-Moser spaces}}}

\author{Gwyn Bellamy}
\address{School of Mathematics and Statistics, University of Glasgow, University Gardens, Glasgow G12 8QW}
\email{gwyn.bellamy@glasgow.ac.uk}

\author{Victor Ginzburg}\address{
Department of Mathematics, University of Chicago,  Chicago, IL 
60637, USA.}
\email{ginzburg@math.uchicago.edu}

\begin{abstract}
We study the natural $GL_2$-action  on the Hilbert scheme of points in
the plane, resp. $SL_2$-action  on the Calogero-Moser space. We describe the closure
of the $GL_2$-orbit, resp. $SL_2$-orbit, of  each point fixed by the 
corresponding diagonal torus.
We also find the character of the representation of the group $GL_2$
in the fiber of the  Procesi bundle, and its Calogero-Moser analogue,
over the $SL_2$-fixed point.
\end{abstract}

\maketitle

\section{Introduction}\label{sec:intro}

\subsection{} 


The natural action of the group $GL_2$ on $\C^2$ induces
a $GL_2$-action on  $\Hilb^n \C^2$,  the Hilbert scheme  of $n$ points
in the plane.
There is also a similar action of the group  $SL_2$ on $\cX_{\bc}$,
the Calogero-Moser space. The fixed points of the
corresponding maximal torus $\C^*\times \C^*$, resp. $\C^*$, of diagonal 
matrices, are labeled by partitions. Let $y_\la\in \Hilb^n \C^2$,
resp. $x_\la\in \cX_{\bc}$, denote the  point labeled
by a partition $\la$. 
It turns out that such a point is fixed by the  group 
$SL_2$ if and only if $\la = (m,m-1,\ds,2,1) =: \stair$ is a \textit{staircase}
partition. In the Hilbert scheme case, this has been observed by Kumar and Thomsen \cite{KumarHilbert}. The case of the  Calogero-Moser space can be
deduced from the Hilbert scheme case using ``hyper-K\"ahler rotation''.
A different, purely algebraic proof is given in section 3 below.

The theory of  rational Cherednik algebras gives an $SL_2 \times \s_n$-equivariant vector bundle $\mc{R}$ of rank $n!$ on the
Calogero-Moser space.
Thus,  $\mc{R}|_{x_{\stair}}$, the fiber of  $\mc{R}$  over the $SL_2$-fixed point,
acquires the structure of a $SL_2 \times \s_n$-representation.
We find  the character formula of this representation in terms  of
Kostka-Macdonald polynomials.
The vector bundle  $\mc{R}$ is an analogue  of  the Procesi bundle $\mc{P}$, a  $GL_2 \times \s_n$-equivariant vector bundle of
rank $n!$ on $\Hilb^n \C^2$. Our  formula agrees
with  the character of the representation of  $GL_2 \times \s_n$
in  $\mc{P}|_{y_{\stair}}$, the fiber of $\mc{P}$ over the $GL_2$-fixed point,
obtained by Haiman \cite{HaimanSurvey}. It is, in fact, possible to
derive our character  formula for  $\mc{R}|_{x_{\stair}}$
from the one for   $\mc{P}|_{y_{\stair}}$. However, 
 the character  formula for  $\mc{P}|_{y_{\stair}}$, as well as the construction
of the Procesi bundle itself,
involves the $n!$-theorem.

In section 2 we review some general results about $SL_2$-actions.
In section 3, we apply these results to show that, for any $\la$, the $SL_2$-orbit
of $x_\la$ is closed in  $\cX_{\bc}$.
The $GL_2$-orbit of $y_\la$ is not closed in  $\Hilb^n \C^2$, in general,
and we describe the closure in \S 4. 

\subsection{Acknowledgments}
The work of the first author was partially supported by the EPSRC grant EP/N005058/1. The work of the second author  was supported in part by the NSF grant DMS-1303462.

\section{$\slt$-actions}
Let $T \subset SL_2$ be the maximal torus of diagonal matrices.
The group $T$ acts on the Lie algebra $\slt$ by conjugation.
Let $(E,H,F)$ be the standard basis of $\slt$. 

Let $X$ be an algebraic variety equipped with
a $T$-action and let $\Vect(X)$ be the
the Lie algebra of algebraic vector fields on $X$.
The $T$-action on $X$ induces a $T$-action on $\Vect(X)$
by Lie algebra automorphisms. An algebraic variety $X$ equipped with
a $T$-action and with a Lie algebra homomorphism
$\slt \to \Vect(X)$ that intertwines the $T$-actions
on $\slt$ and on $\Vect(X)$, respectively,
will be referred to as an $\sltt$-variety.

Given a group $G$ and a $G$-variety $X$,
we write $X^G$ for the fixed point set of $G$.
Given an $\sltt$-variety $X$ we write
$X^{\slt}$ for the closed subset, with reduced scheme structure, of $X$, defined
as the zero locus of all vector fields contained in
the image of the map $\slt \to \Vect(X)$.
Clearly, we have $X^{\slt}\sset X^T$. Any variety with an $SL_2$-action has an obvious structure of an
$\sltt$-variety. In such a case we have $X^{SL_2}=X^{\slt}$.

\begin{thm}\label{thm:fixedA} Let $X$ be  smooth quasi-projective
  variety
equipped with an $\sltt$-action. Then,

\vi If  $x \in X^T$ is an isolated fixed point, then $x\in X^{\slt}$ 
 if and only if all the weights of $T$ on $T_x X$ are odd.  

\vii If the $\sltt$-action on $X$ comes from a
nontrivial $SL_2$-action with dense orbit then the set
$X^{SL_2}$ is finite.
\end{thm}

\begin{proof}
(i) Let $x \in X^{\slt}$ and let $\mf{m}$ be the maximal ideal in the
local ring $\mc{O}_{X,x}$ defining this point. Then $\mf{sl}_2$ acts on $\mf{m}/\mf{m}^2$. Since $x$ is a isolated fixed point for the $T$-action, the degree zero weight space is $0$ so all $\mf{sl}_2$-modules appearing in $\mf{m}/\mf{m}^2$ must have odd weight spaces only. 

Conversely, assume that all non-zero weight spaces in $\mf{m} / \mf{m}^2$ have odd weight. We need to show that $\mf{sl}_2$ acts in this case i.e. $
\mf{sl}_2(\mf{m}) \subset \mf{m}$.  It is known that any $T$-orbit is contained in an affine
$T$-stable Zariski open subset of $X$. Therefore, replacing
$\mc{O}_{X,x}$ by some affine $T$-stable neighborhood, we may assume
that $X$ is an affine $T$-variety with $\mf{sl}_2$-action and isolated
fixed point defined by $\mf{m} \lhd \C[X]$. Then $\C[X] = \C 1 \oplus
\mf{m}$ as a $T$-module. In particular, every homogeneous element of
non-zero degree belongs to $\mf{m}$. If $z \in \mf{m}$ is homogeneous of
degree $\neq -2$ then $\deg E(z) = \deg z + 2 \neq 0$. Thus, $E(z) \in
\mf{m}$. On the other hand, if $\deg z = -2$ then our assumptions imply
that $z \in \mf{m}^2$ and hence $E(z) \in \mf{m}$. A similar argument
applies for $F$.  

Part (ii) is a result of Bialynicki-Birula, \cite[Theorem 1]{BBSL2}.
\end{proof}

Let $N(T)$ be the normalizer of $T$ in $SL_2$. The Borel of upper-triangular matrices in $SL_2$ is denoted $B$. Its opposite is $B^-$.  

\begin{lem}\label{lem:onehomo}
Let $\mc{O}$ be a one-dimensional homogeneous $SL_2$-space. Then $\mc{O} \simeq SL_2 / B$. 
\end{lem}

\begin{proof}
Let $K = \Stab_{SL_2}(x)$ for some $x \in \mc{O}$, a closed subgroup of $SL_2$. Let $\mf{k}$ be the Lie algebra of $K$. Since $\dim \mf{k} = 2$ it is a solvable subalgebra of $\slt$. Therefore it is conjugate to $\mf{b}$. Without loss of generality, $\mf{k} = \mf{b}$. This means that $K^{\circ} = B \subset K \subset N_{SL_2}(B) = B$. 
\end{proof}

\begin{lem}\label{lem:Tstab}
Let $\mc{O}$ be an $SL_2$-orbit in an affine variety $X$. Assume that
the stabilizer of $x \in \mc{O}$ contains $T$. Then $\mc{O}$ is closed
in $X$ and $\Stab_{SL_2}(x)$ is one of:\
 $T,N(T)$ or $SL_2$.  
\end{lem}

\begin{proof}
Let $K = \Stab_{SL_2}(x)$ and $\mf{k} = \Lie K$. Then $\mf{t} \subset \mf{k}$ implies that either $\mf{k} = \slt$ (and hence $\mc{O} = \{ x \}$), $\mf{k} \simeq \mf{b}$ or $\mf{k} = \mf{t}$. We assume that $\mf{k} \neq \slt$.

Let $Y:= \overline{\mc{O}} \sminus
\mc{O}$ and assume $Y$  is nonempty. Then $Y$ is affine and $\dim Y <
\dim \mc{O} =2$, i.e. $\dim Y \leq 1$. Assume first $\dim Y=1$. Since $ \overline{\mc{O}}$ is affine, it has a
unique closed orbit $\mc{O}'$. Therefore $Y$ also has a unique closed
orbit $\mc{O}'$ and all other orbits in $Y$ have dimension greater than
$\dim \mc{O}'$. In particular, $Y$ contains at most $1$ orbit of
dimension zero. Therefore, $Y$ must contain an orbit of dimension
one. The stabilizer $K$ of such an orbit has dimension $2$ and, without
loss of generality, contains $T$. This implies that $\mathrm{Lie} \ K$
is a Borel subalgebra of $\mf{sl}_2$. Hence $K^{\circ}$ is a Borel
subgroup of $SL_2$. Since $K \sset N_{SL_2}(K^{\circ}) = K^{\circ}$, 
it follows that $K=K^{\circ}$ is a Borel subgroup.
This is impossible since $SL_2/K$  does not embedded into an affine variety. Thus, the only remaining possibility is that $\dim Y=0$. Let $Z$  be the normalization of $\overline{\mc{O}}$. Thus, we have an open dense embedding $\mc{O} \into Z$ such that $\dim (Z \setminus \mc{O})=0$. Since both $\mc{O}$ and $Z$ are affine, this is impossible: it would imply that any regular function on $\mc{O}$ extends to $Z$, so the the restriction  $\C[Z] \rightarrow \C[\mc{O}]$ is an isomorphism, hence $\mc{O}=Z$.
\end{proof}

\begin{lem}\label{lem:sltcomplete}
Let $X$ be a complete $SL_2$-variety and $\mc{O}$ an orbit such that the stabilizer of $x \in \mc{O}$ equals $T$, resp. $N(T)$. 
\begin{enumerate}
\item There is a finite (surjective) equivariant morphism $\mathbb{P}^1 \times \mathbb{P}^1 \twoheadrightarrow \overline{\mc{O}}$, resp. $\mathbb{P}^2 \twoheadrightarrow \overline{\mc{O}}$, which is the identity on $\mc{O}$. 
\item This morphism is an isomorphism if and only if $\overline{\mc{O}}$ is normal. 
\item In all cases $\overline{\mc{O}} \smallsetminus \mc{O} \simeq \mathbb{P}^1$ and $\overline{\mc{O}}^{SL_2} = \emptyset$.    
\end{enumerate}
\end{lem}

\begin{proof}
We explain how the lemma can be deduced from the results of \cite{Mabuchi}. 

Matsuchima's Theorem implies that $\mc{O}$ is
affine. Therefore, by \cite[Corollaire 21.12.7]{EGA4part4}, the complement $Y =  \overline{\mc{O}} \setminus
\mc{O}$ has pure codimension one. By Theorem \ref{thm:fixedA} (ii),
there are only finitely many zero-dimensional orbits in $Y$. Therefore
Lemma \ref{lem:onehomo} implies that each irreducible component $Y_i$ of
$Y$ (being one-dimensional) must contain an orbit $\simeq SL_2 /
B$. Since this orbit is complete, it is closed in $Y_i$ i.e. $Y_i \simeq
SL_2 / B$. Moreover, this implies that $Y_i \cap Y_{j} = \emptyset$ for
$i \neq j$ and hence $\overline{\mc{O}}^{SL_2}  = Y^{SL_2} = \emptyset$. 

By \cite[Theorem 5.1]{Mabuchi}, $\mathbb{P}^1 \times \mathbb{P}^1$ (in the case $\mc{O} \simeq SL_2 / T$), resp. $\mathbb{P}^2$ (in the case $\mc{O} \simeq SL_2 / N(T)$),  is the the unique normal completion of $\mc{O}$. In both these cases the complement is a single copy of $\mathbb{P}^1$. Since we have shown that the boundary $Y$ is a finite union of codimension one orbits, the lemma follows from the Observation of \cite[Section 0]{Mabuchi}. 

\end{proof}

\section{Calogero-Moser spaces} 

Let $(W,\h)$ be a finite Coxeter group, with $S$ the set of \textit{all} reflections in $W$ and $\bc : S \rightarrow \C$ a conjugate invariant function. For each $s \in S$, we fix eigenvectors $\alpha_s \in \h^*$ and $\alpha_s^{\vee} \in \h$ with eigenvalue $-1$. Associated to this data is the rational Cherednik algebra
$\H_{\bc}(W)$ at $t = 0$. It is the quotient of the skew group ring $T^* (\h \oplus \h^*) \rtimes W$ by the relations 
$$
[y,x] = -\sum_{s \in S} \bc(s) \frac{\alpha_s(y) x(\alpha_s^{\vee})}{\alpha_s(\alpha_s^{\vee})}, \quad \forall \ x \in \h^*, y \in \h
$$
and $[x,x'] = [y,y'] = 0$ for $x,x' \in \h^*$ and $y,y' \in \h$. We choose  a $W$-invariant inner product $( - , -)$ on $\h$. The form defines an $W$-isomorphism $\h^* \iso \h$, $x \mapsto \check{x}$.

\subsection{}\label{sec:Hcentre} The center $Z(\H_{\bc}(W))$ of $\H_{\bc}(W)$ has a natural Poisson structure, making $\H_{\bc}(W)$ into a Poisson module. Let $x_1, \ds, x_n$ be a basis of $\h^*$ and $y_1, \ds, y_n$ dual basis. Then the elements
\begin{equation}\label{eq:sltH}
E = -\frac{1}{2}\sum_i x_i^2, \quad F = \frac{1}{2}\sum_i y_i^2, \quad H = \frac{1}{2} \sum_{i} x_i y_i + y_i x_i. 
\end{equation}
are central and form an $\slt$-triple under the Poisson bracket. Their action on $\H_{\bc}(W)$ is given by $[E,x] = [F,\check{x}] = 0$, $[E,\check{x}] = x$, $[F,x] = \check{x}$ and $[H,x] = x$, $[H,\check{x}] = - \check{x}$. Their action on $\H_{\bc}(W)$ is locally finite. Therefore this action can be integrated to get a locally finite action of $SL_2(\C)$  on $\H_{\bc}(W)$ by algebra automorphisms.  Explicitly, this action is given on generators by 
$$
\left( \begin{array}{cc}
a & b \\
c & d 
\end{array} \right) \cdot x = a x + c \check{x}, \quad  
\left( \begin{array}{cc}
a & b \\
c & d 
\end{array} \right) \cdot \check{x} = b x + d \check{x},  \quad  
\left( \begin{array}{cc}
a & b \\
c & d 
\end{array} \right) \cdot w = w, \quad \forall x \in \h^*, \ w \in W. 
$$
The Calogero-Moser space $\cX_{\bc}(W)$ is an affine variety
defined as $\Spec Z(\H_{\bc}(W))$. The action of $SL_2(\C)$ restricts to $Z(\H_{\bc}(W))$
and induces a Hamiltonian action on $\cX_{\bc}(W)$, such that its differential is the action of $\mf{sl}_2$ given by the vector fields $\{ E, - \}$, $\{ F, - \}$ and $\{ H, - \}$. 

There are only finitely many $T$-fixed points on $\cX_{\bc}(W)$. When the Calogero-Moser space is smooth, the $T$-fixed points are naturally labeled $x_{\lambda}$, with $\lambda \in \Irr (W)$. These fixed points are uniquely specified by the fact that the simple head $L(\lambda)$ of the baby Verma module $\Delta(\lambda)$ is supported at $x_{\lambda}$; see \cite{Baby} for details.

Consider the element $w_0 = \left( \begin{array}{cc}
0 & 1 \\
-1 & 0 
\end{array} \right)$ in $SL_2$. It normalizes $T$. 

\begin{lem}\label{lem:sgntwist}
Assume that $\cX_{\bc}(W)$  is smooth. Let $x_{\lambda} \in \cX_{\bc}(W)$ be the $T$-fixed point labeled by the representation $\lambda \in \Irr (W)$. Then $w_0 \cdot x_{\lambda}$ is the fixed point labeled by $\lambda \o \sgn$, where $\sgn$ is the sign representation. 
\end{lem}

\begin{proof}
 The automorphism of $\H_{\bc}(W)$ defined by $w_0$ is the Fourier transform, of  order $4$. The fixed point $w_0 \cdot x$ is the support of ${}^{w_0} L(\lambda)$. Thus, it suffices to show that ${}^{w_0} L(\lambda) \simeq L(\lambda \otimes \sgn)$. This is a standard result. 
\end{proof}

\begin{defn}
A $(\H_{\bc},\mf{sl}_2)$-module $M$ is both a left $\H_{\bc}(W)$-module and left $\slt$-module such that the morphism $\H_{\bc}(W) \o M \rightarrow M$ is a morphism of $\slt$-modules. 
\end{defn}

Every finite dimensional $(\H_{\bc}(W),\mf{sl}_2)$-module is set-theoretically supported at a $SL_2$-fixed point. However, not every finite dimensional $\H_{\bc}(W)$-module set-theoretically supported at a $SL_2$-fixed point has a compatible $\mf{sl}_2$-action. 

Let $e$ denote the trivial idempotent in $\C W$. Then $e$ is $SL_2$-invariant and hence $\H_{\bc}(W) e$ is a $(\H_{\bc},\mf{sl}_2)$-module. Thinking of $\H_{\bc}(W) e$ as a finitely generated $Z(\H_{\bc}(W))$-module, we get a $SL_2 \times W$-equivariant coherent sheaf $\mc{R}$ on $\cX_{\bc}(W)$. When the latter space is smooth, $\mc{R}$ is a vector bundle of rank $|W|$. 

\subsection{Type $A$} Let $\H_{\bc}$ be the rational Cherednik algebra
for the symmetric group $\s_n$ at $t = 0$ and $\bc \neq 0$. In this case both
the set of  $T$-fixed points in the CM-space
$\cX_{\bc} := \cX_{\bc}(\s_n)$
 and the set of 
(isomorphism classes of) simple
irreducible representations of $\s_n$  are labeled by partitions of $n$. We
write  $\mf{m}_{\lambda}$ for the maximal ideal of the  $T$-fixed point corresponding to a partition $\la$.

\begin{notation} From now on, the staircase partition $(m,m-1,\ds,
1)$ will be denoted $\stair$.  Given a partition $\lambda$, the corresponding representation of the symmetric group will be denoted $\pi_{\lambda}$. The finite dimensional, irreducible $SL_2$-module with highest weight $m \ge 0$ will be denoted $V(m)$. 
\end{notation}

\begin{equation}\label{eq:Young}
\Yboxdim18pt
\young(7531,531:,31::,1:::)
\end{equation}
Let $x$ be a box of the partition $\lambda$. The \textit{hook length} $h(x)$ of $x$ is the number boxes strictly to the right of $x$ plus the number strictly below plus one. In the above staircase partition the entry of the box is the corresponding hook length. The \textit{hook polynomial} of $\lambda$ is defined to be 
$$
H_{\lambda}(q) = \prod_{x \in \lambda} (1 - q^{h(x)}).
$$  
Let $(q)_n = \prod_{i = 1}^n (1 - q^i)$ and denote by $n(\lambda)$ the
partition statistic $\sum_{i \ge 1} (i-1) \lambda_i$. 

We write $\chi_T$ for the character of a finite dimensional $T$-representation.
 
\begin{lem}\label{lem:Atang}
Let $x_{\lambda}$ be the $T$-fixed point of $\cX_{\bc}$ labeled by the partition $\lambda$. Then 
$$
\chi_{T}(T_{x_{\lambda}} \cX_{\bc}) = \sum_{x \in \lambda}
q^{h(x)} + q^{-h(x)}. 
$$
\end{lem}

\begin{proof} 
It is known that the graded multiplicity of $\pi_{\lambda}$ in the coinvariant ring $\C[\h] / \langle \C[\h]_+^W \rangle$ is given by $(q)_{(n)} q^{ n(\lambda)} H_{\lambda}(q)^{-1}$,
the so called ``fake polynomial''. If we decompose $T_y \cX_{\bc} = (T_y \cX_{\bc})^+ \oplus (T_y \cX_{\bc})^-$ into its positive and negative weight parts, then Theorem 4.1 and Corollary 4.4 of \cite{BellEnd} imply that 
$$
\chi_{T}\left((T_y \cX_{\bc})^+\right) = \sum_{x \in \lambda} q^{ h(x)},
\quad \textrm{ since } \quad \chi_{T} \left( \C \left[ (T_y \cX_{\bc})^+\right] \right) = \frac{1}{H_{\lambda}(q)}.
$$  
The fact that $T$ preserves the symplectic form on $\cX_{\bc}$ implies that $\chi_{T}\left((T_y \cX_{\bc})^-\right) = \sum_{x \in \lambda} q^{-h(x)}$. 
\end{proof}

The following observation is elementary. 

\begin{lem}\label{lem:fixedA}
Let $\lambda$ be a partition such that every hook length in $\lambda$ is odd. Then $\lambda$ is a staircase partition. 
\end{lem}

Lemma \ref{lem:fixedA}, together with Lemma \ref{lem:Atang} and Theorem \ref{thm:fixedA} imply that $SL_2$-fixed points in $\cX_{\bc}$ are very rare. Namely,

\begin{thm}\label{thm:fixedA1}
If $n = \frac{m(m+1)}{2}$, for some integer $m$, then $\cX_{\bc}^{\slt} = \{ x_{\stair} \}$. Otherwise, $\cX_{\bc}^{\slt} = \emptyset$. 
\end{thm}

The lemma, together with Theorem \ref{thm:fixedA} implies

\begin{prop}
There exists a finite dimensional $(\H_c,\mf{sl}_2)$-module if and only
if $n = \frac{m(m+1)}{2}$ for some $m$. In this case, any such module
$M$ is set-theoretically supported at the fixed point $x_{\stair}$ labeled by the staircase partition.
\end{prop}

\begin{proof}
If $M$ is a $(\H_{\bc},\mf{sl}_2)$-module, then its set-theoretic support must be $SL_2$-stable. If $M$ is also finite dimensional, then this support is a finite collection of points. These points must be $SL_2$-fixed since the group is connected. The result follows from Theorem \ref{thm:fixedA1}. 

Finally, we must show that there exists at least one $(\H_{\bc},
\mf{sl}_2)$-module supported at $x_{\stair}$. Let $\mf{m} \lhd
Z(\H_{\bc})$ be the maximal ideal of $x_{\stair}$. Then $\{
\mf{sl}_2, \mf{m} \} \subset \mf{m}$. Recall that the $\H_{\bc}$-module $\H_{\bc} e$ is an
$(\H_{\bc}, \mf{sl}_2)$-module. Thus, $\H_{\bc} e / \mf{m} \H_c e$ is a (simple) $(\H_{\bc}, \mf{sl}_2)$-module supported at $x_{\stair}$. 
\end{proof}

Recall that there is a unique simple $\H_{\bc}$-module $L(\lambda)$ supported at each of the $T$-fixed points $x_{\lambda}$. Notice that we have shown, 

\begin{cor}
The simple module $L(\stair) \simeq \H_{\bc} e / \mf{m}_{\stair} \H_{\bc} e$ is a $(\H_{\bc},\mf{sl}_2)$-module. 
\end{cor}

Equivalently, the above arguments show that $\slt$ acts on the fiber $\mc{R}_{\stair}$ of $\mc{R}$ at $x_{\stair}$. The formula for the character of the tangent space of $\cX_{\bc}(\s_n)$ at $x_{\stair}$ given by Lemma \ref{lem:Atang} shows that 
\begin{equation}\label{eq:sl2tang}
T_{x_{\stair}} \cX_{\bc} \simeq V(m) \otimes V(m-1),
\end{equation}
as $SL_2$-modules. 

Next we describe the $SL_2$-orbits $\mc{O}_{\lambda} := SL_2 \cdot x_{\lambda}$ of the $T$-fixed points $x_{\lambda}$. First, we note that Lemma \ref{lem:Tstab} implies that 

\begin{lem}\label{lem:CMstab}
The orbit $\mc{O}_{\lambda}$ is closed and $\Stab_{SL_2}(x_{\lambda})$ is reductive. 
\end{lem}

Lemma \ref{lem:sgntwist}, Theorem \ref{thm:fixedA1} and Lemma \ref{lem:CMstab} imply that

\begin{prop}
Let $\lambda$ be a partition of $n$. Then, one has the following 3 alternatives:
\begin{enumerate}
\item $\lambda \neq \lambda^t$\  and\ $\mc{O}_{\lambda} = \mc{O}_{\lambda^t} \simeq SL_2 / T$; 
\item $\lambda = \lambda^t\neq \stair$\  and\  $\mc{O}_{\lambda} \simeq SL_2 / N(T)$; 
\item $\lambda = \stair$\  and\  $\mc{O}_{\lambda} = \{ x_{\stair} \}$. 
\end{enumerate}
\end{prop}

\subsection{The $SL_2$-structure of $\mc{R}_{\stair}$}

We define the $SL_2$-module
$$
U_m := (V(m-1) \oplus V(m-2)) \otimes \bigotimes_{i = 1}^{m-2} (V(i) \oplus V(i-1))^{\o 2}.
$$

\begin{prop}\label{prop:Lmslt}
There is an isomorphism of  $SL_2$-modules:
\begin{equation}\label{eq:sl2iso}
\mc{R}_{\stair} \simeq [U_m \otimes U_{m-2} \otimes \cdots \otimes
U_{2,1}]^{\oplus \dim \pi_{\stair}}
\end{equation}
where the final term $U_{2,1}$ is either $U_2$ or $U_1$ depending on whether $m$ is even or odd.  
\end{prop}

\begin{proof}
As an $(\H_{\bc},\slt)$-module, $\mc{R}_{\stair}$ equals $\H_{\bc} e /
\mf{m} \H_{\bc} e $. As $\H_{\bc}$-modules, $\H_{\bc} e / \mf{m}
\H_{\bc} e $ is isomorphic to $L(\stair)$. Thus, it suffices to show
that the character of $L(\stair)$ as an $SL_2$-module equals the
character of the right hand side of equation (\ref{eq:sl2iso}). The
character of $L(\stair)$ is given in \cite[Lemma
3.3]{Singular}. However, we must shift the grading on $L(\stair)$ from
the one given in \textit{loc. cit.} so that the isomorphism $\H_{\bc} e/ \mf{m} \H_{\bc} e \rightarrow L(\stair)$ is graded i.e. we require
that the one-dimensional space $e L(\stair)$ lies in degree zero. Then, 
$$
\chi_T(L(\stair)) =  q^{-n(\stair)} \frac{H_{\stair}(q)}{(1 - q)^n} \dim \pi_{\stair}.
$$
Note that 
$n(\stair) = \frac{1}{6}(m-1)m(m+1)$.
For the staircase partition, the character of $L(\stair)$ has a natural factorization. The largest hook in $\stair$ is $(m,1^{m-1})$ and $\stair = (m,1^{m-1}) + [m-2]$, therefore peeling away the hooks gives  $q^{-n(\stair)}  / q^{-n([m-2])} = q^{- (m-1)^2}$ and  
\begin{align*}
\frac{H_{\stair}(q)}{(1 - q)^{2m - 1} H_{[m-2]}(q)} & = \frac{1}{(1 - q)^{2m -1}} \left( (1 - q^{2m - 1}) \prod_{i = 1}^{m-1} (1 - q^{2i - 1})^2 \right) \\
 & = \frac{1 - q^{2m-1}}{1 -q} \prod_{i = 1}^{m-1} \left(  \frac{1 - q^{2i - 1}}{1 -q} \right)^2.
\end{align*}
Thus,  
$$
\frac{H_{\stair}(q) q^{- (m-1)^2}}{(1 - q)^{2m - 1} H_{[m-2]}(q)} = (q^{m-1} + q^{m-2} + \cdots + q^{-(m-1)}) \prod_{i = 1}^{m-2} (q^{i} + q^{i-1} + \cdots + q^{-i})^2. 
$$
This is precisely the character of $U_m$. 
\end{proof}

One would like to refine this character by taking into account the action of $W$ too. We decompose $L(\stair)$ as a $W \times SL_2$-module, 
\begin{equation}\label{eq:Ldecomp}
L(\stair) = \bigoplus_{ \la \vdash n} \pi_{\lambda} \otimes V_{\lambda}. 
\end{equation}
Then the \textit{exponents} of $\lambda$ are defined to be the positive integers $0 \le e_{1} \le e_{2} \le \cdots $ such that $V_{\lambda} = \bigoplus_{i} V(e_{i})$. The fact that $L(\stair) $ is the regular representation as a $W$-module implies that  
$$
\dim \pi_{\lambda} = \sum_{i} (e_i + 1) = \dim V_{\la}.
$$
\begin{example}
For $m = 3$, we have $n = 6$ and 
$$
\begin{array}{c|l}
\lambda & e_{1},e_2, \ds \\ \hline 
(6) & 0 \\
(5,1) & 1,2 \\
(4,2) & 1,2,3 \\
(4,1,1) & 0,1,2,3 \\
(3,3) & 0,3 \\
(3,2,1) & 0,1^2,2^2,4 \\
(3,1,1,1) & 0,1,2,3 \\
(2,2,2) & 0,3 \\    
(2,2,1,1) & 1,2,3  \\
(2,1,1,1,1) & 1,2 \\
(1,1,1,1,1,1) & 0  
\end{array}
$$
\end{example}


\begin{lem}\label{lem:exponentduality}
The exponents of $\lambda$ equal the exponents of $\lambda^t$. 
\end{lem}

\begin{proof}
There is an algebra isomorphism $\sgn : \H_{\bc} \iso \H_{-\bc}$ defined by $\sgn(x) = x$, $\sgn(y) = y$ and $\sgn(w) = (-1)^{\ell(w)} w$, where $x \in \h^*, y \in \h$, $w \in \s_n$ and $\ell$ is the length function. It is clear from (\ref{eq:sltH}) that $\sgn$ is $SL_2$-equivariant. Moreover ${}^{\sgn} L(\lambda) \simeq L(\lambda^t)$. In particular, ${}^{\sgn} L(\stair) \simeq L(\stair)$. This isomorphism maps $V_{\lambda}$ to $V_{\lambda^t}$ since ${}^{\sgn} \pi_{\lambda} \simeq \pi_{\lambda} \o \sgn \simeq \pi_{\lambda^t}$. 
\end{proof}

Using the deeper combinatorics of Macdonald polynomials, we prove

\begin{prop}\label{schur}
$\chi_T(V_{\lambda}) = \widetilde{K}_{\lambda,\stair}(q,q^{-1})$.
\end{prop}

\begin{proof}
Let $s_{\lambda}$ denote the Schur polynomial associated to the partition $\lambda$ so that $s_{\lambda} \left[ \frac{Z}{1 - q} \right]$ is a particular plethystic substitution of $s_{\lambda}$; we refer the reader to \cite{HaimanSurvey} for details. 

The module $L(\stair)$ is a graded quotient of the Verma module $\Delta(\stair) = \H_{\bc}(W) \o_{\C[\h^*] \rtimes W} \pi_{\stair}$. The graded $W$-character of $\Delta(\stair)$ is given by $s_{\stair} \left[ \frac{Z}{1 - q} \right]$. As shown in \cite{Baby}, the graded multiplicity of $L(\stair)$ in $\Delta(\stair)$ is given by 
$$
(q)_n^{-1} q^{-n(\stair)} f_{\stair}(q) = H_{\stair}(q)^{-1} = \prod_{i = 1}^m (1 - q^{2i -1})^{-(m - i)}
$$
Therefore, the graded $W$-character, shifted by $q^{-n(\stair)}$ so that $eL(\stair)$ is in degree zero, of $L(\stair)$ equals $q^{-n(\stair)} H_{\stair}(q) s_{\stair} \left[ \frac{Z}{1 - q} \right]$. This implies that 
\begin{equation}\label{eq:schursub}
\chi_T(V_{\lambda}) = \left\langle s_{\mu}, q^{-n(\stair)} \prod_{i = 1}^m (1 - q^{2i -1})^{m - i} s_{\stair} \left[ \frac{Z}{1 - q} \right] \right\rangle.
\end{equation}
The fact that the right hand side of (\ref{eq:schursub}) equals $\widetilde{K}_{\lambda,\stair}(q,q^{-1})$   follows from the property of transformed Macdonald polynomials, \cite[Proposition 3.5.10]{HaimanSurvey}. 
\end{proof}

\begin{rem}
A similar analysis can be done for other Coxeter groups $W$. For instance, when $W$ is a Weyl group of type $B$ and $\bc$ generic, it is easily seen that $\cX_{\bc}(W)^{\slt} \equiv \emptyset$. 
\end{rem}

\section{The Hilbert scheme of points in the plane}\label{sec:Hilbert}

The group $SL_2$ also acts naturally on the Hilbert scheme $\Hilb^n \C^2$ of $n$ points in the plane. This is the restriction of a $GL_2$-action, induced by the natural action of $GL_2$ on $\C^2$. 

\subsection{} The $T$-fixed points $y_{\lambda}$ in $\Hilb^n \C^2$ are also labeled by partitions $\lambda$ of $n$. If $I$ is the $T$-fixed, codimension $n$ ideal labeled by $\lambda$, then it is uniquely defined by the fact that the corresponding quotient $\C[x,y] / I_{\lambda}$ has basis given by $x^{i} y^j$ with 
$$
(i,j) \in Y_{\lambda} := \{ (i,j) \in \Z^2 \ | \ 0 \le j \le \ell(\lambda) - 1, \ 0 \le i \le \lambda_j - 1\}, 
$$
the \textit{Young tableau} of $\lambda$. The orbit $GL_2 \cdot y_{\lambda}$ is denoted $\mc{O}_{\lambda}$. Identify $\Cs$ with the scalar matrices in $GL_2$. Then $(\Hilb^n \C^2)^{\Cs}$ is the moduli space of homogeneous ideals of codimension $n$ in $\C[x,y]$, as studied in \cite{Iarrobino}. It is a smooth, projective $GL_2$-stable subvariety of $\Hilb^n \C^2$, containing the points $y_{\lambda}$. Notice that the $GL_2$-orbits and $SL_2$-orbits in $(\Hilb^n \C^2)^{\Cs}$ agree since the action factors through $PGL_2$. 

\begin{lem}
If $n = \frac{m(m+1)}{2}$, for some integer $m$, then $(\Hilb^n \C^2)^{GL_2} = \{ y_{\stair} \}$. Otherwise, $(\Hilb^n \C^2)^{GL_2} = \emptyset$. 
\end{lem}

\begin{proof}
This follows from  \cite[Lemma 12]{KumarHilbert}. Alternatively, notice that if $y_{\lambda}$ is fixed by $GL_2$, then $\C[x,y] / I_{\lambda}$ is an $GL_2$-module. Since each graded piece of $\C[x,y]$ is an irreducible $GL_2$-module, this implies that there is some $m$ such that $I_{\lambda} = \C[x,y]_{\ge m}$ and hence $\lambda = \stair$. 
\end{proof}

We say that a partition $\lambda$ is \textit{steep} if $\lambda_1 > \cdots > \lambda_{\ell} > 0$. 

\begin{prop} \label{prop:HilbSL2}
Let $\lambda \neq \stair$ be a partition of $n$ and set $K = \Stab_{SL_2}(y_{\lambda})$. 
\begin{enumerate}
\item If $\lambda$ is steep then $K = B$, and if $\lambda^t$ is steep
  then $K = B_-$. In both cases, $\mc{O}_{\lambda} \simeq
  \mathbb{P}^1$. 

\item If neither $\lambda$ or $\lambda^t$ is steep, then $K = T$ if $\lambda \neq \lambda^t$ and $K = N(T)$ if $\lambda = \lambda^t$. In both cases the complement to $\mc{O}_{\lambda}$ in $\overline{\mc{O}_{\lambda}}$ equals $\mathbb{P}^1$.
\item The orbit $\mc{O}_{\lambda}$ is closed if and only if $\lambda$ or $\lambda^t$ is steep. 
\end{enumerate}
\end{prop}

\begin{proof}
If $\lambda$ is steep then \cite[Lemma 12]{KumarHilbert} shows that $B \subset K$. If $\dim K > \dim B$, then $\dim K = 3$ i.e. $K = SL_2$ and $\lambda = \stair$ (notice that $\stair$ is the only partition such that both $\lambda$ and $\lambda^t$ are steep). Therefore $\dim B = \dim K$ and hence $K^{\circ} = B$. But then $N_{SL_2}(B) = B$ implies that $K = B$. Since $y_{\lambda^t} = w_0 \cdot y_{\lambda}$, if $\lambda^t$ is steep then $K = w_0 B w_0^{-1} = B_-$. This proves part (1). 

Assume now that neither $\lambda$ nor $\lambda^t$ are steep. Let $\Lie K = \mf{k}$. Since $\mf{k} \supset \mf{t}$, but $\mf{k} \not\simeq \mf{b}, \slt$, we have $\mf{k} = \mf{t}$ and hence $K = T$ or $N(T)$. Then part (2) follows from Lemma  \ref{lem:sltcomplete}. Notice that  Lemma  \ref{lem:sltcomplete} is applicable here even though $\Hilb^n \C^2$ is not complete; this is because $\mc{O}_{\lambda}$ is contained in the punctual Hilbert scheme $\Hilb^n_0 \C^2 \subset \Hilb^n \C^2$ of all ideals supported at $0 \in \C^2$. This $SL_2$-stable subvariety is complete.   
  
Part (3) follows directly from parts (1) and (2). 
\end{proof}

\begin{question}
For which $\lambda$ is $\overline{\mc{O}}_{\lambda}$ normal?
\end{question}

Associate to a partition $\lambda$ diagonals $d_k := | \{ (i,j) \in Y_{\lambda} \ | \ i+ j = k \}|$, where $k = 0,1,\ds$. For instance, if $\lambda = (4,3,3,1,1)$, then the diagonals $(d_0,d_1,\ds)$ are $(1,2,3,4,2)$. Now construct a new partition $U(\lambda)$ from $\lambda$ by setting $U(\lambda)_i = | \{ d_k \ | \ d_k \ge i \} |$. It is again a partition of $|\lambda|$. Pictorially, we if we visualize the Young tableau $Y_{\lambda}$ in the English style, as in (\ref{eq:Young}), then on the $k$th diagonally (where there are $d_k$ boxes), we have simply moved all boxes as far to the top-right as possible. E.g. $U(4,3,3,1,1) = (5,4,2,1)$. If instead we move all boxes on the $k$th diagonally as far to the bottom left as possible, we get $U(\lambda)^t$. 

\begin{lem}
Let $\lambda$ be a partition. 
\begin{enumerate}
\item The partition $U(\lambda)$ is steep and $U(\lambda) = \lambda$ if and only if $\lambda$ is steep. 
\item $U(\lambda) = \stair$ if and only if $\lambda = \stair$. 

\end{enumerate}
\end{lem}

\begin{proof}
It is clear from the construction that $U(\lambda)$ is steep; if $\lambda_{i-1} = \lambda_{i}$ for some $i$ then one can move the box at the end of $i$th row further up and to the right on the diagonal that it belongs to. Similarly, if $\lambda$ is steep, then $\lambda_{i-1} > \lambda_i$ for all $i$ such that $\lambda_i \neq 0$ implies that there is always a box ``above and to the right'' of a given box i.e. if $(i,j) \in Y_{\lambda}$ and $i \neq 0$ then $(i-1,j+1) \in Y_{\lambda}$ (this can be viewed as an alternative definition of steep). 

Part (2) is also immediate from the construction. 
\end{proof}


\begin{prop}
Let $\lambda$ be a partition such that neither $\lambda$ nor $\lambda^t$ is steep then $\overline{\mc{O}_{\lambda}} = \mc{O}_{\lambda} \sqcup  \mc{O}_{U(\lambda)}$.
\end{prop}

\begin{proof}
Grade $\C[x,y]$ by putting $x$ and $y$ in degree one. Then every $I \in \mc{O}_{\lambda}$ is graded, $I = \bigoplus_{k \ge 0} I_k$ and $\dim I_k$ is independent of $I$. Since $\dim (I_{\lambda})_k = k + 1 - d_k$, we deduce that $\dim I_k = k+1-d_k$ for all $I \in \mc{O}_{\lambda}$.  By Proposition \ref{prop:HilbSL2} (2) and Lemma \ref{lem:sltcomplete}, we know that $\overline{\mc{O}_{\lambda}} = \mc{O}_{\lambda} \sqcup  \mc{O}'$, where $\mc{O}' \simeq SL_2 / B$. Thus, there exists a steep partition $\mu \neq \stair$ such that $\mc{O}' = \mc{O}_{\mu}$.  

 The Hilbert-Mumford criterion implies that there exists some $I \in
 \mc{O}_{\lambda}$ such that $J = \lim_{t \rightarrow 0} t \cdot I$ is a
 $T$-fixed point in $\mc{O}_{\mu}$. Thus, either $J = I_{\mu}$ or $J =
 I_{\mu^t}$. Without loss of generality, $J = I_{\mu}$. This implies
 that $\dim (I_{\mu})_k = k + 1 - d_k$. Since $\mu$ is steep,
 $(I_{\mu})_{k}$ is a $B$-submodule of $\C[x,y]_k$, cf. Proposition
\ref{prop:HilbSL2} (1). Therefore,  $\{ x^k, x^{k-1} y, \ds, x^{k+1 - d_k} y^{d_k-1} \}$ is a basis of $(\C[x,y]  / I_{\mu})_k$ i.e. $\{ (i,j) \in Y_{\mu} \ | \ i+ j = k \}$ equals $\{ (k,0),(k-1,1), \ds, (k+1-d_{k},d_k - 1) \}$. But $U(\lambda)$ is uniquely defined by this property. Hence $\mu = U(\lambda)$.

\end{proof}

\begin{rem}
For any (homogeneous) ideal $I \in (\Hilb^n \C^2)^{\Cs}$, $I$ is fixed by $B$ if and only if each $I_k$ is a $B$-submodule of $\C[x,y]_k$. But the $B$-submodules of $\C[x,y]_k$ are the same as the $U$-submodules of $\C[x,y]_k$. This implies that $I$ is $B$-fixed if and only if it is $U$-fixed. 
\end{rem}

It is known, see eg \cite[Theorem 5.6]{GordonSmith},  that the Hilbert scheme fits into a flat family $p : \mf{X} \rightarrow \mathbb{A}^1$ such that $p^{-1}(0) \simeq \Hilb^n \C^2$ and $p^{-1}(\bc) \simeq \cX_{\bc}$ for $\bc \neq
0$. Moreover, $SL_2$ acts on $\mf{X}$ such that the map $p$ is equivariant, with $SL_2$ acting trivially on $\C$. The identification of the fibers is also equivariant. The set-theoretic fixed point set $\mf{X}^T$ decomposes
$$
\mf{X}^T = \bigsqcup_{\lambda \vdash n} \mathbb{A}_{\lambda}, 
$$
into a union of connected components $\mathbb{A}_{\lambda}$, where $\mathbb{A}_{\lambda} \simeq \mathbb{A}^1$ with $p^{-1}(\bc) \cap \mathbb{A}_{\lambda}= \{ x_{\lambda} \}$ for $\bc \neq 0$ and $p^{-1}(0) \cap \mathbb{A}_{\lambda} = \{ y_{\lambda} \}$. The only thing that is not immediate here is that the parameterization of the fixed points in $\cX_{\bc}$ match those of $\Hilb^n \C^2$. But this can be seen from Lemma \ref{lem:Atang}, \cite[Lemma 5.4.5]{HaimanSurvey} and the fact that a partition is uniquely defined by its hook polynomial. 

Then the $SL_2$-varieties $SL_2 \cdot \mathbb{A}_{\lambda}$ are connected. Assume that neither $\lambda$ nor $\lambda^t$ is steep. Then there are equivariant trivializations 
$$
SL_2 \cdot \mathbb{A}_{\lambda} \simeq SL_2 / N(T) \times \mathbb{A}^1 \quad \textrm{or} \quad SL_2 \cdot \mathbb{A}_{\lambda} \simeq SL_2 /T \times \mathbb{A}^1,
$$
depending on whether $\lambda = \lambda^t$ or not. 

Let $\widetilde{\slt} \rightarrow \slt$ be Grothendieck's simultaneous
resolution and write $\varpi$ for the composition
$\widetilde{\slt} \rightarrow \slt\to \slt/\!/SL_2\cong \mathbb{A}^1$,
where the second map is $a\mto \frac{1}{2}\Tr a$.

\begin{conj}
Let $\lambda \neq \stair$ be a steep partition. There exists a
$SL_2$-equivariant embedding $\widetilde{\slt} \hookrightarrow \mf{X}$
sending the $B$-fixed point $[1:0] \in \mathbb{P}^1 \subset
\widetilde{\slt}$ to $y_{\lambda}$ and such that the following
diagram commutes
$$\xymatrix{
\widetilde{\slt}\ar[dr]_<>(0.5){\varpi}
 \ar@{^{(}->}[rr] && \mf{X}\ar[dl]^<>(0.5){p}\\
&\mathbb{A}^1&
}
$$
\end{conj}

\subsection{The Procesi bundle}\label{sec:Procesi} The Procesi bundle $\mc{P}$ on $\Hilb^n \C^2$ is a $GL_2 \times \s_n$-equivariant vector bundle of rank $n !$. See \cite{HaimanSurvey}, and references therein, for details. The fiber $\mc{P}_{\stair}$ is a $GL_2 \times \s_n$-module, decomposing as 
$$
\mc{P}_{\stair} = \bigoplus_{\mu \vdash n} V_{\mu} \o \pi_{\mu}. 
$$
As $GL_2$-modules, we have a decomposition $V_{\mu} = \bigoplus_i V (m_{i},n_{i})$ into a direct sum of irreducible $GL_2$-modules $V(m_{i},n_{i})$ with highest weight $(m_{i},n_{i})$; here $m_{i},n_{i} \in \Z$, with $m_{i} \ge n_{i}$. We call $(m_{1},n_{1}), (m_{2},n_{2}), \ds $ the \textit{graded exponents} of $\mu$. Let $H$ denote the $2$-torus of diagonal matrices in $GL_2$. The character of $V_{\mu}$ is given by the cocharge Kostka-Macdonald polynomial,
\begin{equation}\label{eq:characterco}
\chi_H(V_{\lambda}) = \widetilde{K}_{\lambda,\stair}(q,t). 
\end{equation}
Notice that this implies $\widetilde{K}_{\lambda,\stair}(q,t) = \widetilde{K}_{\lambda,\stair}(t,q)$. This can also be deduced directly from the definition of Macdonald polynomials e.g. \cite[Proposition 3.5.10]{HaimanSurvey}. Similarly, equation (\ref{eq:characterco}), together with standard properties \cite[Proposition 3.5.12]{HaimanSurvey} of Macdonald polynomials imply that 
$$
V_{\lambda^t} \simeq V_{\lambda}^* \otimes \mathrm{det}^{\otimes n(\stair)}.
$$
Thus, if the exponents of $\lambda$ are $(m_1,n_1),\ds $ then the exponents of $\lambda^t$ are
\[
(n(\stair) - n_1, n(\stair) - m_1), \ds
\]

\begin{question}
Is there an explicit formula for the graded exponents of $\lambda$?
\end{question} 


Next we explain how Lemma \ref{lem:exponentduality} and Proposition \ref{schur} can be deduced from the statements of section \ref{sec:Procesi}, \textit{provided} one uses Haiman's $n!$ Theorem. 

Let $u$ be a formal variable and $\H_{u \bc}$ the flat $\C[u]$-algebra such that $\H_{u \bc} / \langle u \rangle \simeq \H_0$ and $\H_{u \bc} / \langle u - 1 \rangle \simeq \H_{\bc}$. By \cite[Theorem 5.5]{GordonSmith}, the space $\mf{X}$ can be identified with a moduli space of $\lambda$-stable $\H_{u \bc}$-modules $L$ such that $L |_{\s_n} \simeq \C \s_n$. Here $\lambda$ is a generic stability parameter; see \textit{loc. cit.} for definitions. As such, $\mf{X}$ comes equipped with a canonical bundle $\widetilde{\mc{P}}$ such that each fiber is a $\H_{u \bc}$-module. The action of $SL_2$ on $\mf{X}$ lifts to $\widetilde{\mc{P}}$. 

\begin{thm}
For $\bc \neq 0$, $\widetilde{\mc{P}} |_{p^{-1}(\bc)} \simeq \mc{R}$ and $\widetilde{\mc{P}} |_{p^{-1}(0)} \simeq \mc{P}$. 
\end{thm}

\begin{proof}
The first claim follows from \cite[Section 3]{EG} and the second is a consequence of Haiman's proof of the $n !$-conjecture; see the proof of \cite[Theorem 5.3]{GordonSmith} and references therein. 
\end{proof}

\begin{cor}
As $\s_n \times SL_2$-modules, $\mc{R}_{\stair} \simeq \mc{P}_{\stair}$ and hence $\chi_T(V_{\lambda}) = \chi_H(V_{\lambda}) |_{t = q^{-1}}$. 
\end{cor}



\small{

}
\end{document}